\documentclass[12pt, letterpaper]{article}
\usepackage{amsfonts}
\usepackage{amssymb}
\usepackage{latexsym}  
\usepackage{pifont}  
\usepackage{bbm}  %

\let\SavedRightarrow=\Rightarrow
\usepackage{marvosym}
\let\Rightarrow=\SavedRightarrow

\usepackage{amsmath}   

\newenvironment{itemizz}{\begin{itemize}\setlength{\itemsep}{-1mm}}
{\end{itemize}}

{\begin{itemize} \setlength{\itemsep}{-1mm} %
\renewcommand\labelitemi{\ding{#1}}} %
{\end{itemize}}

\newtheorem{theorem}{Theorem}[section]
\newtheorem{definition}[theorem]{Definition}
\newtheorem{lemma}[theorem]{Lemma}

\newtheorem{corollary}[theorem]{Corollary}

\newcommand\RRR{\mathbb {R}}
\newcommand\PPP{\mathbb {P}}
\newcommand\QQQ{\mathbb {Q}}

\newcommand\DD{\mathcal {D}}

\newcommand\FF{\mathcal {F}}

\newcommand\AP{\mathcal{AP}}

\newcommand\dom{\mathrm{dom}}  
\newcommand\ran{\mathrm{ran}}   
\newcommand\cl{\mathrm{cl}}  
\newcommand\lh{\mathrm{lh}}   
\newcommand\MA{\mathrm{MA}}  
\newcommand\PFA{\mathrm{PFA}}  
\newcommand\ZFC{\mathrm{ZFC}}  
\newcommand\CH{\mathrm{CH}}  
\newcommand\hgt{\mathrm{ht}}   

\newcommand\AV{\mathsf{A\! V}} 

\newcommand\cchi{{\raise 2 pt \hbox{$\chi$}}}

\newcommand\res{\mathord {\upharpoonright}}  

\newcommand\cat{^{\mathord{\frown}}}  

\newcommand\iv{^{-1}} 

\newcommand\one{\mathbbm{1}} 

\newcommand\V{{\mathbf V}}  

\newcommand\compat{\mathrel{\not\perp}}  

\newcommand\pre[2]{ {}^{#1} #2 }
\newcommand\stp{ {}^{*}\!p }

\newcommand\eop{{\Large \Coffeecup}}  

\newenvironment{proof}{{\bf Proof.}}{\eop\medskip}
\newenvironment{proofof}[1]{\medskip \textbf{Proof of #1.}}{\eop\medskip}

\begin{document}

\title{Forcing and Differentiable Functions%
}

\author{Kenneth Kunen\footnote{University of Wisconsin,  Madison, WI  53706, U.S.A.,
\ \ kunen@math.wisc.edu} }

\maketitle

\begin{abstract}
We consider covering $\aleph_1 \times \aleph_1$ rectangles
by countably many smooth curves, and differentiable isomorphisms
between $\aleph_1$-dense sets of reals.
\end{abstract}

\section{Introduction}
\label{sec-intro}
In this paper, we consider two different issues, both related to the
question of obtaining differentiable real--valued functions
where classical results only produced functions or continuous functions.

Regarding the first issue,
the text of Sierpi\'nski \cite{Sier} shows that $\CH$
is equivalent to his Proposition $P_2$, which is
the statement that the plane
``est une somme d'une infinit\'e d\'enombrable
de courbes''.    Here, a ``curve'' is just the graph of 
a function or an inverse function, so $P_2$
says only that $\RRR^2 = \bigcup_{i \in \omega} (f_i \cup f_i\iv)$,
where each $f_i$ is (the graph of) a function from $\RRR$ to $\RRR$,
with no assumption of continuity.  The proof 
actually shows, in $\ZFC$, that for every $E \in [\RRR]^{\aleph_1}$, there 
are $f_i : \RRR \to \RRR$ with
$E^2 \subseteq \bigcup_{i \in \omega} (f_i \cup f_i\iv)$,
and that this is false for all $E$ of size greater than $\aleph_1$.

Usually in geometry and analysis, ``curve'' \emph{does} imply continuity,
so it is natural to ask whether
the $f_i$ can all be continuous, or even $C^\infty$:

\begin{definition}
For $n \in \omega \cup \{\infty\}$, call
$E \subseteq \RRR$ \emph{$n$--small} iff there are $C^n$ functions
$f_i: \RRR \to \RRR$ such that
$E^2 \subseteq \bigcup_{i \in \omega} (f_i \cup f_i\iv)$.
Here, $C^0$ just means ``continuous'', and 
$C^\infty$ means $C^n$ for all $n \in \omega$.
\end{definition}

Countable sets are trivially $\infty$--small,
and by Sierpi\'nski, $|E| \le \aleph_1$ for every $0$--small set $E$,
so we are only interested in sets of size $\aleph_1$.
Every $0$--small set is of first category
and measure $0$ (and perfectly meager and universally null).  
Just in $\ZFC$, we shall prove the following in Section \ref{sec-smallness}:

\begin{theorem}
\label{thm-small}
There is an $E \in [\RRR]^{\aleph_1}$ which is $\infty$--small.
\end{theorem}

The existence of a $0$--small set is due to
Kubi\'s and Vejnar \cite{Kubis}.

But now we can ask whether \emph{every}
$E \in [\RRR]^{\aleph_1}$ is $n$--small for some $n$.
Even when $n = 0$, this would imply that every such $E$ is
of first category
and measure $0$ (and perfectly meager and universally null),
which is a well-known consequence of $\MA(\aleph_1)$.
In fact, the following theorem follows easily from results
already in the literature, as we shall point out in 
Section \ref{sec-smallness}:

\begin{theorem}
\label{thm-all-small}
{\qquad  } 
\begin{itemizz}
\item[1.] $\MA(\aleph_1)$ implies that every set of size 
$\aleph_1$ is $0$--small.  
\item[2.] $\PFA$ implies that every set of size 
$\aleph_1$ is $1$--small.
\item[3.] $\MA(\aleph_1)$ does not imply that every set of size 
$\aleph_1$ is $1$--small.
\item[4.] In $\ZFC$, 
there is an $E \in [\RRR]^{\aleph_1}$ which is not $2$--small.
\end{itemizz}
\end{theorem}

We remark that Sierpi\'nski's use of ``curve'' is unusual in
another way:  Usually, we would call a subset of $\RRR^2$
a curve iff it is a continuous image of $[0,1]$,
and not necessarily the graph of a function; but with
that usage, the plane is always a countable union of curves
by Peano \cite{Peano}.

\bigskip

Our second issue 
involves the isomorphism of $\aleph_1$-dense
subsets of $\RRR$.  

\begin{definition}
$E \subseteq \RRR$ is \emph{$\aleph_1$-dense} iff 
$|E \cap (x,y)| = \aleph_1$ whenever $x,y \in \RRR$ and $x < y$.
$\FF$ is the set of all order-preserving bijections from 
$\RRR$ onto $\RRR$.
\end{definition}

By Baumgartner \cite{Ba1,Ba2}, $\PFA$ implies that whenever $D,E$ are 
$\aleph_1$-dense, there is an $f \in \FF$ such that $f(D) = E$.
By \cite{AS, ARS}, this cannot
be proved from $\MA(\aleph_1)$ alone.
Clearly, every $f \in \FF$ is continuous, but we can ask
whether we can always get our $f$ to be $C^n$.

For $n = 2$, a $\ZFC$ counter-example is apparent from
Theorems \ref{thm-small} and \ref{thm-all-small},
since we may take $D$ to be $2$--small and $E$ to be not $2$--small,
and also assume that $D = D + \QQQ = \{ D + q : q \in \QQQ \}$
and $E = E + \QQQ$.  Note that $D$ is $2$--small iff $D + \QQQ $
is $2$--small, and the latter set is also $\aleph_1$-dense.

But in fact, even $n = 1$ is impossible, since the following holds
in $\ZFC$, as we shall show in Section \ref{sec-non-isom}:

\begin{theorem}
\label{thm-not-equiv}
There are $\aleph_1$-dense $D, E \subset \RRR$ such that
for all $f \in \FF$ and $\aleph_1$-dense $D^* \subseteq D$ and
$E^* \subseteq E$ with $f(D^*) = E^*$:
If $p < q$ and $a = f(p)$ and $b = f(q)$ then:
\begin{itemizz}
\item[1.] Either $f$ is not uniformly Lip\-schitz on $(p,q)$
\emph{or} $f\iv$ is not uniformly Lip\-schitz on $(a,b)$; equivalently,
whenever $0 < \Lambda  \in \RRR$, there are $x_0, x_1 \in (p,q)$
such that either $|f(x_1) - f(x_0)| > \Lambda |x_1 - x_0|$
\emph{or} $|x_1 - x_0| > \Lambda |f(x_1) - f(x_0)|$.
\item[2.] Either $f'$ does not exist at some $d \in D^* \cap (p,q)$ 
\emph{or} $(f\iv)'$  does not exist at some $e \in E^* \cap (a,b)$.
\item[3.] If $f'(d)$ exists for all $d \in D^* \cap (p,q)$, then $f'(d) = 0$
for all but countably many $d \in D^* \cap (p,q)$.
\end{itemizz}
\end{theorem}

In particular, $f$ cannot be in $C^1(\RRR)$, since $f'$ cannot
be $0$ everywhere, so if $f'$ were continuous, there would
be an interval on which $f' > 0$, contradicting (3).

On the other hand, $f'$ can exist everywhere and be $0$ on
a dense set if $f'$ is not required to be continuous:

\begin{theorem}
\label{thm-almost-c1}
Assume $\PFA$, and let $D,E \subset \RRR$ be $\aleph_1$-dense.
Then there exist $f \in \FF$ and 
$D^* \subseteq D$ such that $D^* $ is $\aleph_1$-dense
and $f(D^*) =  E$ and
\begin{itemizz}
\item[1.] For all $x \in \RRR$, $f'(x)$ exists and $0 \le f'(x) \le 2$.
\item[2.] $f'(d) = 0$ for all $d \in D^*$.
\end{itemizz}
\end{theorem}

By (1),  $f$ satisfies a uniform Lipschitz condition
with Lipschitz constant $2$.    The ``$2$'' is an artifact of the proof,
and may be replaced by an arbitrarily small number;
if $\varepsilon > 0$, we can get our $f$ with $f'(x) \le 2$
so that $f(D^*) = (1/\varepsilon) E$;
then $\varepsilon f' (x) \le 2 \varepsilon$ and
$\varepsilon f(D^*) =  E$.
In (2), the $f'(d) = 0$ is to be expected, in view of
Theorem \ref{thm-not-equiv}(3).
We do not know whether we can make $D^*$ equal $D$.

The proof of Theorem \ref{thm-almost-c1}
in Sections \ref{sec-edif}  and \ref{sec-isom} actually shows that one
can force the result to hold in an appropriate ccc extension
of any model of $\ZFC + 2^{\aleph_0} = \aleph_1 + 2^{\aleph_1} = \aleph_2 $.
Then the result follows from $\PFA$ using the same forcing plus the
``collapsing the continuum'' trick.

We remark that Theorem \ref{thm-almost-c1} contradicts
Proposition 9.4  in the paper \cite{ARS}
of  Abraham,  Rubin, and  Shelah, which produces a
$\ZFC$ example of $\aleph_1$-dense $D, E \subset \RRR$ such that
every $f \in \FF$ with $f \cap (D \times E)$ uncountable
fails to be differentiable at uncountably many elements of $D$.
Their ``proof'' uses ideas similar to our proof of
Theorem \ref{thm-not-equiv}, but insufficient details
are given to be able to locate a specific error.

\section{On Smallness}
\label{sec-smallness}
We first point out that
Theorem \ref{thm-all-small} follows easily from known results:

\begin{proofof}{Theorem \ref{thm-all-small}}
For (1), fix $E \in [\RRR]^{\aleph_1}$.
By Sierpi\'nski,
$E^2 \subseteq \bigcup_{i \in \omega} (f_i \cup f_i\iv)$,
where each $f_i$ is the graph of a function and $|f_i| = \aleph_1$.
Then, assuming $\MA(\aleph_1)$,
a standard forcing shows that for each $i$,
there are Cantor sets $P_{i,n}$ for $n \in \omega$
with each $P_{i,n}$ the graph of a function and
$f_i \subseteq \bigcup_n P_{i,n}$.
Now each $P_{i,n}$ extends to a function 
$g_{i,n} \in C(\RRR,\RRR)$, so that
$E^2 \subseteq \bigcup_{i,n } (g_{i,n} \cup g_{i,n}\iv)$.

For (2), use the fact from \cite{HK} that under PFA,
every $A \in [\RRR^2]^{\aleph_1}$ is a subset of a
countable union of $C^1$
arcs.  Now apply this with $A = E \times E$, and note that
every $C^1$ arc is contained in a finite union of (graphs of) $C^1$
functions and inverse functions.

(4) also follows from \cite{HK}, which shows in $\ZFC$ that there
is an  $A \in [\RRR^2]^{\aleph_1}$ which is not a subset of a countable union
of $C^2$ arcs.  So, choose $E$ such that $A \subseteq E \times E$.

Likewise, (3) follows from \cite{Kun}, which shows that it
is consistent with $\MA(\aleph_1)$ to have
an  $A \in [\RRR^2]^{\aleph_1}$ which is a weakly Luzin set; and such a
set is not a subset of a countable union of $C^1$ arcs.
\end{proofof}

Next, to prove Theorem \ref{thm-small}, we first state
an abstract version of the argument involved:

\begin{lemma}
\label{lemma-coding}
Suppose that $T$ is an uncountable set with functions $f_i$ on $T$  for 
$i \in \omega$ such that for all countable $Q \subset T$,
there is an $x \in T$ such that $Q \subseteq \{f_i (x) : i \in \omega\}$.
Then there is an 
$E \subseteq T$ of size $\aleph_1$ such that
$E \times E \subseteq \Delta \cup \bigcup_{i} (f_i\cup f_i\iv)$,
where $\Delta$ is the identity function.
\end{lemma}
\begin{proof}
Note, by considering supersets of $Q$, that there must be uncountably
many such $x$.  Now, let $E = \{e_\alpha : \alpha < \omega_1\}$
where $e_\alpha$ is chosen recursively so that
$e_\alpha \notin \{e_\xi  : \xi < \alpha\}
\subseteq \{f_i (e_\alpha) : i \in \omega\}$.
\end{proof}

To illustrate the idea of our argument, we first produce
an $E \in [\RRR]^{\aleph_1}$ which is $0$--small,
in which case $T$ can be any Cantor set.

\begin{lemma}
\label{lemma-cov}
There are $f_i \in C(2^\omega, 2^\omega)$ for $i < \omega$
such that for all countable non-empty $Q \subseteq 2^\omega$,
there is an $x \in 2^\omega$ such that
$Q = \{f_i(x) : i < \omega\}$.
\end{lemma}
\begin{proof}
Let $\varphi$ map
$\omega\times\omega$ 1-1 into $\omega$, and 
let $(f_i(x))(j) = x(\varphi(i,j))$.
Now, let $Q = \{y_i : i \in \omega\}$.
Since $\varphi$ is 1-1, we may choose $x \in 2^\omega$
such that $x(\varphi(i,j)) = y_i(j)$ for all $i,j$;
then $f_i(x) = y_i$.
\end{proof}

So, if $T\subseteq \RRR$ is a Cantor set, then $T \cong 2^\omega$,
and the existence of an $E \in [T]^{\aleph_1}$ which is $0$--small
follows from Lemmas \ref{lemma-coding} and \ref{lemma-cov},
and the observation that every function in $C(T,T)$ extends
to a function in $C(\RRR,\RRR)$.

Now, if we want our functions to be smooth, as required by
Theorem \ref{thm-small}, we must be a bit more careful.
The $f_i$ will be defined on the standard middle-third Cantor set $H$,
but they will only 
satisfy Lemma \ref{lemma-coding} on a thin subset $T \subset H$.

To simplify notation,  $H$ will be a subset of $[0,3]$ rather than $[0,1]$.
For $x \in [0,3]$, 
$x \in H$ iff $x$ has only $0$s and $2$s in its ternary expansion,
so that $x =  \sum_{n\in\omega} x(n) 3^{-n}$, where each $x(n) \in \{0,2\}$,
and we write $x$ in ternary as $x(0) . x(1) x(2) x(3) x(4) \cdots$. 
If $x,y \in H$ with $x \ne y$,
let $\delta(x,y)$ be the least $n$ such that $x(n) \ne y(n)$, and note that
$3^{-n} \le |x - y| \le 3^{-n + 1}$.

Fix any $\Gamma : \omega \to \omega$ 
such that $\Gamma(0) = 0$,
$\Gamma$ is strictly increasing, and $\Gamma(k+1) \ge (\Gamma(k))^2$
for each $k$.
The minimum such $\Gamma$ is the sequence $0,1,2,4,16,256, \ldots$,
but any other such $\Gamma$ will do.

We view $x$ in $H$ as coding an $\omega$--sequence of \emph{blocks},
where  the $k^\mathrm{th}$  block
is a sequence of length $\Gamma(k+1) - \Gamma(k)$.
Note that
$\Gamma(k+2) - \Gamma(k + 1) \ge \Gamma(k+1) - \Gamma(k)$ for each $k$,
so the blocks get longer as $k \nearrow$.

More formally,
for $x \in H$ and $k \ge 0$, we define $B^x_k : \omega \to \{0,2\}$ 
so that $B^x_k(j) = x(\Gamma(k)  + j)$ when
$j < \Gamma(k+1) - \Gamma(k) $ and $B^x_k(j) = 0$ for 
$j \ge \Gamma(k+1) - \Gamma(k)$.
Note that $x$ is determined by $\langle B^x_k : k \in \omega \rangle$.
Let $B^x_k(j) = 0$ when  $k < 0$.

Now, we wish $x \in H$ to encode a sequence of $\omega$ element of $H$,
$\langle f_i(x) : i \in \omega \rangle$. 
We do this using a bijection $\varphi$
from $\omega \times \omega$ onto $\omega$. 
We assume that 
$\max(i,j) < \max(i',j') \rightarrow \varphi(i,j) < \varphi(i',j')$
for all $i,j,i',j'$,
which implies that $\max(i,j)^2 \le   \varphi(i,j) <  (\max(i,j) + 1)^2$.

In the ``standard'' encoding,
as in the proof of Lemma \ref{lemma-cov},
an $x \in \{0,2\}^\omega$ can encode
$\omega$ elements of $\{0,2\}^\omega$, where the $i^\mathrm{th}$ element
is $j \mapsto x(\varphi(i,j))$.  But here, for $x \in K$, we apply 
this separately to each of the $\omega$ blocks of $x$,
\emph{and} we shift right two places to ensure that the functions are smooth.
Define $f_i : H \to H$ so that for $x \in H$, $f_i(x)$ is
the $z \in H$ such that 
$B^z_k(j) = B^x_{k-2}(\varphi(i,j))$ for all $j$;
so $B^z_k(j) = 0$ when $k < 2$.  There is such a $z$ because
\[
j \ge \Gamma(k+1) - \Gamma(k ) \;\Rightarrow\;
\varphi(i,j)  \ge j \ge \Gamma(k-1) - \Gamma(k - 2) \;\Rightarrow\;
B^z_k(j) = 0 \ \ .
\]
Let $S = \{0,2\}^{<\omega}$.
For $i \in \omega$ and $s \in S$,
define $f_i^s : H \to H$ so that for $x \in H$, $f_i^s(x)$ is
the $z \in H$ such that $z(n)$ is $s(n)$ for $n < \lh(s)$ and
$f_i(x)(n)$ for $n \ge \lh(s)$.

Note that most elements of $H$ are not in
$\bigcup \{f_i^s(H) : i \in \omega \ \&\ s \in S\}$,
but the $T$ of Lemma \ref{lemma-coding} will be a proper subset of $H$.

First, we verify that we get $C^\infty$ functions.
Following \cite{HK}, call $f : H \to H$ \emph{flat} iff
for all $q \in \omega$,
there is a bound $M_q$ such that for all $u,t \in H$,
$| f(u) - f(t) | \le M_q |u - t|^q$.
By Lemma 6.4 of \cite{HK}, this implies that $f$ can be extended
to a $C^\infty$ function defined on all of $\RRR$,
all of whose derivatives vanish on $H$.

\begin{lemma}
Each $f_i^s $ is flat.
\end{lemma}
\begin{proof}
Fix $x,y$ in $H$ with $x \ne y$.  Let $n = \delta(x,y)$.
Fix $k \in \omega$ so that $\Gamma(k) \le n < \Gamma(k+1)$.
Then $\Gamma(k + 2) \le \delta(f_i^s(x), f_i^s(y))$.
Now $|x - y| \ge 3^{-n} \ge 3^{- \Gamma(k + 1)}$,
and $|f_i^s(x) - f_i^s(y)| \le  3^{- \Gamma(k + 2) + 1}$, so
\[
|f_i^s(x) - f_i^s(y)| / |x - y| ^q \le
3^{ - \Gamma(k+2) + 1 + q \Gamma(k+ 1)} \le \ \ 
3^{ - (\Gamma(k+1))^2 + 1 + q \Gamma(k+ 1)}  \ \ ,
\]
which is bounded, and in fact goes to $0$ as $k \nearrow \infty$.
\end{proof}

Now, we define $T \subset H$:  For $x \in H$ and $k \in \omega$,
let  $\ell_k^x$ be the least $\ell \in \omega$ such that
$\forall j \ge \ell \, [B^x_k(j) = 0]$.
So, $\ell_k^x \le  \Gamma(k+1) - \Gamma(k)$.

Call $\psi : \omega \to \omega$ \emph{tiny}
iff $\lim_{k \to \infty} (\psi(k)^n)/k = 0 $ for all $n \in \omega$.
Note that tininess is preserved by powers and shifts.  That is,
if $\psi$ is tiny, then so is $k \mapsto \psi(k)^r$ and
$k \mapsto r + \psi(k + r)$ for each $r > 0$.

\begin{proofof}{Theorem \ref{thm-small}}
Let $T$ be the set of $x \in H$ such that $k \mapsto \ell_k^x$ is tiny.
Then $T$ is an uncountable Borel set, and we are done by
Lemma \ref{lemma-sequence}:
\end{proofof}

\begin{lemma}
\label{lemma-sequence}
If $y_i \in T$ for $i \in \omega$, then there is an 
$x \in T$ and $s_i \in S$ for $i \in \omega$ such that
$f_i^{s_i}(x) = y_i$ for all $i$.
\end{lemma}
\begin{proof}
Fix any $\psi : \omega \to \omega$ such that
$\psi(k) \le \Gamma(k+1) - \Gamma(k) $ for all $k$.  Then we can
define $x \in H$ so that $  B^x_{k}(\varphi(i,j)) = B^{y_i}_{k+2}(j) $
whenever $\varphi(i,j) < \psi(k)$; let  $B_x^{k}(m) = 0$ for 
$m \ge \psi(k) $.  Then $x \in T$ provided that $\psi$ is tiny.

For each $i$, the function $k \mapsto (i + \ell_{k+2}^{y_i})^2$ is tiny.
Now, fix a tiny $\psi$ such that
$\psi(k) \le \Gamma(k+1) - \Gamma(k)$ for all $k \in \omega$
and $\psi \ge^* (k \mapsto (i + \ell_{k+2}^{y_i})^2)$
for each $i$; this is possible by a standard diagonal argument.

Now fix $i$.   Then fix $r \in \omega$ such that
$\psi(k) \ge (i + \ell_{k+2}^{y_i})^2$ for all $k \ge r$.
Let $s_i = y_i \res \Gamma(r + 2)$.
Let $z = f_i^{s_i}(x)$.  We shall show that $z = y_i$.
So, fix $n \in \omega$, and we show that $z(n) = y_i(n)$.
This is obvious if $n < \Gamma(r + 2)$, so assume that 
$n \ge \Gamma(r + 2)$.
Then fix $k \ge r+2$ and $j < \Gamma(k+1) - \Gamma(k)$
with $n = \Gamma(k) + j$.  We must show that
$B^z_{k}(j) = B^{y_i}_{k}(j)$.

By definition of $f_i^{s_i}$,
$B^z_k(j) = B^x_{k-2}(\varphi(i,j))$,
whereas we only know that $B^{y_i}_{k}(j) =  B^x_{k-2}(\varphi(i,j))$
when $\varphi(i,j) < \psi(k-2)$.
So, assume that  $\varphi(i,j) \ge \psi(k-2)$; we show that
$B^z_k(j) = 0$ and $ B^{y_i}_{k}(j) = 0$.

Now $ B^{y_i}_{k}(j) = 0$ because otherwise
$j < \ell_{k}^{y_i}$, and then
$\varphi(i,j) \le (i + j)^2 < (i + \ell_{k}^{y_i})^2 \le \psi(k-2)$,
a contradiction.

Also, $B^z_k(j) = B^x_{k-2}(\varphi(i,j)) = 0$ by
the  definition of $x$, since $\varphi(i,j) \ge \psi(k-2)$.
\end{proof}

\section{Non-Isomorphisms}
\label{sec-non-isom}
Here we prove Theorem \ref{thm-not-equiv}.
First,

\begin{lemma}
\label{lemma-cantor}
There are Cantor sets $H,K \subset \RRR$ such that
\begin{align*}
& \forall \varepsilon > 0 \; \exists \delta > 0 \;
\forall x_0, x_1 \in H \; \forall y_0, y_1 \in K  \; \\
&\qquad \big[ 0 < |x_1 - x_0| < \delta \; \wedge \;  0 < |y_1 - y_0| < \delta \;
\longrightarrow \\
&\qquad\qquad (y_1 - y_0)/(x_1 - x_0) \in
(-\varepsilon, \varepsilon) \cup 
(1/\varepsilon, \infty) \cup (-\infty, - 1/\varepsilon) \big] \ \ .
\end{align*}
\end{lemma}
\begin{proof}
We obtain $H,K$ by the usual trees of closed intervals:
\begin{itemizz}
\item[1.] $H = \bigcap_{n\in\omega} \bigcup\{I_\sigma :
\sigma \in \pre{n}{2}\}$
and $K = \bigcap_{n\in\omega} \bigcup\{J_\tau : \tau \in \pre{n}{2}\}$.
\item[2.] $I_\sigma = [a_\sigma, b_\sigma]$ and $J_\tau = [c_\tau, d_\tau]$.
\item[3.]
$a_\sigma = a_{\sigma\cat 0} <  b_{\sigma\cat 0} <
 a_{\sigma\cat 1} <  b_{\sigma\cat 1} = b_\sigma$.
\item[4.]
$c_\tau = c_{\tau\cat 0} <  d_{\tau\cat 0} <
 c_{\tau\cat 1} <  d_{\tau\cat 1} = d_\tau$
\item[5.]
Whenever $\lh(\sigma) = \lh(\tau) = n$:  
$b_\sigma - a_\sigma = p_n$ and $b_\tau - a_\tau = q_n$.
\end{itemizz}

Informally, assume that $\lh(\sigma) = \lh(\tau) = n$.  Then $I_\sigma \times J_\tau$
is a box of dimensions $p_n \times q_n$.  It will be very long
and skinny ($p_n \gg q_n$).
Inside this box will be four little boxes, of dimensions
$p_{n+1} \times q_{n+1}$, situated at the corners of the $p_n \times q_n$ 
box.  These little ones are much smaller; that is,
$p_n \gg q_n \gg p_{n+1} \gg q_{n+1}$.
Now suppose that the two points $(x_0,y_0)$ and $(x_1,y_1)$ both lie
in $I_\sigma \times J_\tau$, but lie in different smaller boxes
$I_{\sigma\cat \mu} \times J_{\tau \cat \nu}$.  So, there are 
$\binom{4}{2} = 6$ possibilities. 
For two of them, between 
$I_{\sigma\cat \mu} \times J_{\tau \cat 0}$ and
$I_{\sigma\cat \mu} \times J_{\tau \cat 1}$  ($\mu \in \{0,1\}$),
the slope $|\Delta y / \Delta x|$
is very large.  For the other four, between
$I_{\sigma\cat 0} \times J_{\tau \cat \nu}$ and
$I_{\sigma\cat 1} \times J_{\tau \cat \nu}$  ($\nu \in \{0,1\}$),
or between
$I_{\sigma\cat 0} \times J_{\tau \cat 0}$ and
$I_{\sigma\cat 1} \times J_{\tau \cat 1}$ or between
$I_{\sigma\cat 0} \times J_{\tau \cat 1}$ and
$I_{\sigma\cat 1} \times J_{\tau \cat 0}$,
$|\Delta y / \Delta x|$ is very small.

More formally, assume that $p_0 > q_0 > p_1 > q_1 > \cdots $ 
and $q_n/ p_n \to 0$ and $p_{n+1}/q_n \to 0$ as $n \to \infty$.
Fix  $(x_0,y_0)$ and $(x_1,y_1)$ in $H \times K$, and then fix $n$
such that for some $\sigma,\tau \in \pre{n}2$,
$(x_0,y_0) , (x_1,y_1) \in I_\sigma \times J_\tau$, but 
$(x_0,y_0) , (x_1,y_1)$ are in two different smaller boxes
$I_{\sigma\cat \mu} \times J_{\tau \cat \nu}$.  Note that this $n \to \infty$
as $\delta \to 0$.
In the two large slope cases,
$|\Delta y / \Delta x| \ge (q_n - 2 q_{n+1})/ p_{n+1} \to \infty$
as $n \to \infty$, since $q_n  / p_{n+1} \to \infty$ and 
$ q_{n+1}/ p_{n+1} \to 0$.
In the four small slope cases,
$|\Delta y / \Delta x| \le q_n / (p_n - 2 p_{n+1}) \to 0$,
since $p_n /q_n \to \infty$  and $ p_{n+1} / q_n \to 0$.
\end{proof}

\begin{proofof}{Theorem \ref{thm-not-equiv}}
Fix $H,K$ as in Lemma \ref{lemma-cantor}, 
and then fix $\tilde H \in [H]^{\aleph_1}$ and $\tilde K \in [K]^{\aleph_1}$.
Let 
$D = \bigcup\{ \tilde H + s : s \in \QQQ\}$ and
$E = \bigcup\{ \tilde K + t : t \in \QQQ\}$.

Now, fix $f, D^*, E^*, p,q,a,b$ as in Theorem \ref{thm-not-equiv}.
Then the function
$f^* := f \cap \big((D^* \cap (p,q)) \times (E^* \cap (a,b))\big)$
is uncountable, and is an order-preserving bijection from $D^* \cap (p,q) $
onto $E^* \cap (a,b)$.

Now fix $s,t \in \QQQ$ so that $f^* \cap (\tilde H + s) \times (\tilde K + t)$
is uncountable, so in particular it contains a convergent sequence.
So, we have $(x_n, y_n) \in f^*$ for $n \le \omega$,
with $x_n \to x_\omega$ and $y_n \to y_\omega$ as $n \nearrow \omega$,
and $x_n \in \tilde H + s$ and $y_n \in \tilde K + t$
for all $n \le \omega$.
We may assume that all the $x_n$ are distinct and that
all the $y_n$ are distinct.
Since $f^*$ is order-preserving and
the property of $H,K$ in Lemma \ref{lemma-cantor} is preserved by translation,
$ \forall \varepsilon > 0 \; \exists n\in \omega \; [
(y_\omega - y_n)/(x_\omega - x_n) \in
(0, \varepsilon) \cup (1/\varepsilon, \infty) ]$.
Passing to a subsequence, we may assume that either
$\forall n \in \omega\; [ (y_\omega - y_n)/(x_\omega - x_n) \in (2^n, \infty)]$
or
$\forall n \in \omega\; [ (y_\omega - y_n)/(x_\omega - x_n) \in (0, 2^{-n}) ]$.
In the first case, $f'(x_\omega)$ doesn't exist and $f$ is not
Lipschitz on $(p,q)$.
In the second case, $(f\iv)'(y_\omega)$ doesn't exist and $f\iv$ is not
Lipschitz on $(a,b)$.

For (3), repeat the argument, now letting $f^*$ be the set
of all $(d, f(d))$ such that  $p < d < q$ and $f'(d)$ exists and $f'(d) \ne 0$.
\end{proofof}

\section{Everywhere Differentiable Functions}
\label{sec-edif}
We prove here some lemmas to be used in the proof of
Theorem \ref{thm-almost-c1},  where
we shall construct the isomorphism $f$ along with its derivative $g$.

\begin{definition}
For $g: \RRR \to \RRR$,
let $\|g\| = \sup\{ |g(x)| : x \in \RRR\} \in [0, \infty]$.
\end{definition}

\begin{definition}
$\DD$ is the set of all measurable $g: \RRR \to \RRR$ such that 
$\|g\| < \infty$ and
$g(x) = \lim_{h \to 0} \frac 1 h \int_x^{x+ h} g(t) \, dt$
for all $x$.
\end{definition}

By this last condition, if $f(x) = \int_0^x g(t)\, dt$,
then $f'(x) = g(x)$ for all $x$.

Note that $\DD$ is a Banach space with the $\sup$ norm $\|\cdot\|$.
Also, $\DD$ contains all bounded continuous functions, and every function
in $\DD$ is of Baire class $1$; that is,
a pointwise limit of continuous functions.
However, many Baire $1$ functions,
such as $\cchi_{ \{0\} }$, fail to be in $\DD$.
A function in $\DD$ can be everywhere discontinuous;
this has been known since the 1890s;
see pp.~412--421 of Hobson \cite{Hob} for references.
Katznelson and  Stromberg \cite{KS} describe a method for constructing
such functions which we can embed into our forcing construction.
Here we summarize their method and make some minor additions to it.

\begin{definition}  For $\psi: \RRR \to \RRR$ and  $a \ne b$:
\[
\AV^b_a \psi  = \frac{1}{b-a} \int_a^b \psi(x) \, dx \ \ .
\]
\end{definition}

\begin{definition}
\label{def-AP}
Fix $C > 1$.  $\psi : \RRR \to \RRR$ has the \emph{$C$--average property}
iff $\psi $ is bounded and continuous, and $\psi(x) \ge 0$ for all $x$,
and $ \AV^b_a \psi  \le C \min(\psi(a), \psi(b)) $ whenever $a \ne b$.
Let $\AP_C$ be the set of all functions with the $C$--average property.
\end{definition}

So, the average value of $\psi$ on an interval is bounded by $C$ times  the
value at either endpoint.
Note that either $\psi(x) > 0$ for all $x$ or $\psi = 0$  for all $x$.
Also, $\AP_C$ is closed under finite sums and uniform limits, and
if $\psi\in \AP_C$
then $( x \mapsto \alpha \psi(\beta x + \gamma) ) \in \AP_C$
for all $\alpha, \beta, \gamma \in \RRR$ with $\alpha \ge 0$.
$\AP_C$  clearly contains all non-negative constant functions,
but also, by \cite{KS},
the function $(1 + |x|)^{-1/2}$ has the $4$--average property; see 
also Lemma \ref{lemma-psi-4}  below.  Functions in $\AP_C$ can be used to build
functions in $\DD$ by:

\begin{lemma}
\label{lemma-D-sum}
Fix $C > 1$.  Assume that all $\psi_j \in  \AP_C$.  Let $g(x) = \sum_{j\in\omega} \psi_j(x)$,
and assume that $g(x) < \infty$ for all $x$ and $\|g\| < \infty$.  Then $g  \in \DD$.
\end{lemma}
\begin{proof}
Fix $x \in \RRR$ and $\varepsilon > 0$.  It is sufficient to produce a $\delta > 0$ such that:
\[
\forall h \in (-\delta, \delta) \backslash \{0\} \ : \qquad
g(x) - 2 \varepsilon  \le
\AV^{x + h}_x g \le
g(x) + (C + 1) \varepsilon \ \  .   \tag{$\ast$}
\]
Let $g_m(x) = \sum_{j < m} \psi_j(x)$.  First fix $m$ such that $g_m(x) \ge g(x) - \varepsilon$.
Then fix $\delta > 0$ such that $|g_m(x) - g_m(x+h)| \le \varepsilon$ for all
$h \in (-\delta, \delta) \backslash \{0\}$.  Then, fix such an $h$, and we verify $(\ast)$.
For the first $\le$, use $g(x) -2 \varepsilon  \le g_m(x) - \varepsilon  \le
\AV_x^{x+h} g_m \le  \AV_x^{x+h} g $.  For the second $\le$, note that for each $n \ge m$,
$(g_n - g_m) \in \AP_C$, and hence $\AV_x^{x+h} (g_n - g_m) \le C( g_n(x) - g_m(x)) \le C \varepsilon$.
Letting $n \nearrow \infty$, we get 
 $\AV_x^{x+h} (g - g_m) \le C \varepsilon$,
 so that $\AV^{x + h}_x g \le  \AV^{x + h}_x g_m + C \varepsilon \le
   g_m(x) + (C + 1) \varepsilon \le g(x) + (C + 1) \varepsilon $.
\end{proof}

To verify that the function  $(1 + |x|)^{-1/2}$ has the $4$--average property:

\begin{lemma}
\label{lemma-AP}
Suppose that 
$\psi : \RRR \to [0, \infty)$ is a bounded measurable function
such that $\psi( x) = \psi( -x)$ for all $x$,
$\psi$ is decreasing for $x > 0$,
and $\AV_0^{ b} \psi \le C \psi( b)$ for all $b > 0$. 
Then $\psi \in \AP_{2C}$.
\end{lemma}
\begin{proof}
We must show that $\AV^b_a \psi \le 2C \min(\psi(a), \psi(b))$ whenever
$a \ne b$.  By symmetry, there are only two cases:

Case I:  $a < 0 < b$, where $0 < \hat a := -a \le b$
(so $\psi(a) \ge \psi(b)$):
\[
\AV^b_a \psi  = 
\frac{1}{b + \hat a } \left[ \int_0^b \psi +  \int_0^{\hat a} \psi \right] \le
\frac{1}{b} \cdot 2  \int_0^b \psi  \le 2C \psi(b) \ \ .
\]

Case II:  $0 \le a < b$:  Then, since $\psi$ is decreasing,
$ \AV^b_a \psi \le \AV^b_0 \psi \le C \psi(b)$.
\end{proof}

\begin{lemma}
\label{lemma-psi-4}
If $\psi(x) = (1 + |x|)^{-1/2}$ then $\psi \in \AP_4$.
\end{lemma}
\begin{proof}
For $b > 0$,
\[
\frac{1}{\psi(b)} \AV_0^b \psi =
\frac{\sqrt{1 + b}}{b} \left[ 2 \sqrt{1 + b} - 2 \right ] =
\frac{2}{b} \left[ b + 1 -  \sqrt{1 + b} \right]  <  2 \ \ ,
\]
so apply Lemma \ref{lemma-AP}.
\end{proof}

\section{Isomorphisms}
\label{sec-isom}
This entire section is devoted to the proof of Theorem \ref{thm-almost-c1}.
We plan to construct $f$ along with $g = f'$, which will be in $\DD$;
so $f(x) = \int_0^x g(t) \, dt$.
We shall construct $g$ as a limit of an $\omega$--sequence,
using the following modification of Lemma \ref{lemma-D-sum}:

\begin{lemma}
\label{lemma-limit-D}
Assume that we have $g_n, \psi_n, \theta_n$ for $n \in \omega$
such that:
\begin{itemizz}
\item[1.] $g_0 \in C(\RRR, [0, \infty)\,)$ and $\|g_0\| < \infty$.
\item[2.] $\theta_n \in C(\RRR,\RRR)$, and 
$\sum_{n} \|\theta_n\| < \infty$.
\item[3.] Each $\psi_n \in \AP_4$.
\item[4.] $g_{n+ 1} = g_n - \psi_n + \theta_n$ and 
$g_{n+ 1}(x) \ge 0$ for all $x$.
\end{itemizz}
Then  $\langle g_n : n \in \omega \rangle$ converges pointwise
to some $g : \RRR \to [0, \infty)$, and $g \in \DD$.
\end{lemma}

\begin{proof}
Since all $\psi_i \ge 0$ and all $g_n \ge 0$,
all sums $h_n := \sum_{i < n} \psi_i$,
and hence also $h := \sum_{i < \omega} \psi_i$, are bounded by
$\|g_0\| + \sum_i \|\theta_i\|$.
It follows that the sequence
$\langle g_n : n \in \omega \rangle$ converges pointwise,
and $h \in \DD$ by Lemma \ref{lemma-D-sum}.
Then $g \in \DD$ because $g = g_0 + \sum_{n}\theta_n  - h$
and $ g_0 + \sum_{n}\theta_n \in \DD$ (since it is bounded and continuous).
\end{proof}

We plan to build the $\psi_n$ and $\theta_n$ by forcing,
and the forcing conditions will guarantee that
each $g_n(x) \ge 0$ for all $x$.
Besides $f(x) :=  \int_0^x g(t) \, dt$, we
also have $f_n(x) := \int_0^x g_n(t) \, dt$,  and the
$f_n$ will converge pointwise to $f$.
Since $f(0)$ must be $0$, we shall assume WLOG that $0 \in D \cap E$.
The proof applies the ``collapsing the continuum'' trick;
so we assume $\CH$, and we describe a ccc poset which forces the
$\psi_n$ and $\theta_n$.

To construct ccc posets, we use the standard setup with elementary submodels:

\begin{definition}
\label{def-elem-submod-chain}
Fix $\kappa$,  a suitably large regular cardinal.
Let $\langle M_\xi : 0 < \xi < \omega_1 \rangle$ be a continuous 
chain of countable elementary submodels of $H(\kappa)$,
with $D,E \in M_1$ and each $M_\xi \in M_{\xi + 1}$.
Let $M_0 = \emptyset$.
For $x \in \bigcup_\xi M_\xi$, let $\hgt(x)$, the \emph{height} of $x$,
be the $\xi$ such that $x \in M_{\xi+1} \backslash M_\xi$.
\end{definition}

By setting $M_0 = \emptyset$, we ensure that under $\CH$,
$\hgt(x)$ is defined whenever $x \in \RRR$ or 
$x$ is a Borel subset of $\RRR$.
Observe that $\{d \in D : \hgt(d) = \xi\}$ and
$\{e \in E : \hgt(e) = \xi\}$ are both countable and dense
for each $\xi < \omega_1$.

We now state the basic combinatorial lemma behind the proof
of ccc.  This lemma uses the compatibility symbol $\compat$,
but does not  mention forcing explicitly.

\begin{lemma}
\label{lemma-compat}
Assume $\CH$.  Say we have $2n$--tuples
\[
p^\alpha =
\big( (d^\alpha_0, e^\alpha_0) , \ldots , (d^\alpha_{n-1}, e^\alpha_{n-1})\big) 
\in \RRR^{2n}
\]
for $\alpha < \omega_1$.   Fix $\varphi \in C( (0, \infty), (0, \infty) )$.
Assume that:

\begin{itemizz}
\item[a.]
$d^\alpha_i \ne d^\beta_i$ and $e^\alpha_i \ne e^\beta_i $
for all $\alpha,\beta,i$ with $\alpha \ne \beta$.
\item[b.]
$\hgt(d^\alpha_i) > \hgt(e^\alpha_i)$ for all $\alpha,i$.
\item[c.]
$\hgt(d^\alpha_i) \ne \hgt(d^\alpha_j) $ 
for all $\alpha,i,j$ with $i \ne j$.
\end{itemizz}
Then there exist $\alpha \ne \beta$ with
$p^\alpha \compat p^\beta$, in the sense that for all $i < n$, the slope
$( e^\beta_i - e^\alpha_i) / ( d^\beta_i - d^\alpha_i)  > 0$ 
and also
$| e^\beta_i - e^\alpha_i| < \varphi( |d^\beta_i - d^\alpha_i| ) $.
\end{lemma}

Here, we are asserting that
the two-element partial function
$\{ (d^\alpha_i , e^\alpha_i ) , (d^\beta_i  , e^\beta_i) \}$
is order-preserving, and also has slope bounded by
a ``small'' function $\varphi$.

\medskip

\begin{proof} Induct on $n$.  The case $n = 0$ is trivial,
so assume the result for $n$ and we prove it for $n+1$, so
now $ p^\alpha =
( (d^\alpha_0, e^\alpha_0)  ,  \ldots, (d^\alpha_{n}, e^\alpha_{n}))
\in \RRR^{2n+2}$.  Applying (b)(c), WLOG, each sequence is arranged so
that $\hgt(p^\alpha) = \hgt(d^\alpha_{n})$, and hence
$\hgt(p^\alpha) > \hgt(d^\alpha_{i})$ for all $i < n$ and
$\hgt(p^\alpha) > \hgt(e^\alpha_{i})$ for all $i \le n$.
Also, by (a), WLOG, $\alpha < \beta \to \hgt(p^\alpha) < \hgt(p^\beta)$,
which implies that $\hgt(p^\alpha) \ge \alpha$.

Let $F = \cl\{ p^\alpha : \alpha < \omega_1\} \subseteq \RRR^{2n+2}$.
For each $\alpha$ and each $x \in \RRR$, obtain
$q^\alpha_x \in \RRR^{2n+2}$
by replacing the $d^\alpha_{n}$ by $x$ in $p^\alpha$.
Let $F^\alpha = \{x \in \RRR : q^\alpha_x \in F\}$.
Fix $\zeta$ such that $F \in M_\zeta$.
For $\alpha \ge \zeta$, $F^\alpha$ is uncountable because
$d^\alpha_{n} \in F^\alpha$ and $F^\alpha \in M_{\hgt(p^\alpha)}$
while $d^\alpha_{n} \notin M_{\hgt(p^\alpha)}$.
So, choose any $u^\alpha, v^\alpha \in F^\alpha$ with
$u^\alpha < v^\alpha$.  Then, get an uncountable 
$S \subseteq \omega_1 \backslash \zeta$, along with
rational open intervals $U,V$ such that $\sup U < \inf V$ 
and $u^\alpha \in U$ and $v^\alpha \in V$.
Let $\Xi = \inf\{\varphi(y - x) : y \in V \ \&\ x \in U\}$.
Thinning $S$, we assume also that for $\alpha,\beta\in S$,
$| e^\alpha_n - e^\beta_n | < \Xi$.

Let $\stp^\alpha = p^\alpha \res (2n)$ (delete the last pair).
Applying induction, fix $\alpha,\beta \in S$ such that
$\stp^\alpha \compat \stp^\beta$ and
$e^\alpha_n < e^\beta_n$.
Now,
$q^\alpha_{u^\alpha}, q^\beta_{v^\beta} \in F$,
so we may choose $p^\delta, p^\epsilon$ sufficiently close to 
$q^\alpha_{u^\alpha}, q^\beta_{v^\beta} \in F$, respectively,
such that $\stp^\delta \compat \stp^\epsilon$ and also so that
$d^\delta_n \in U$ and
$d^\epsilon_n \in V$, and also so that
$ 0  <  e^\epsilon_n - e^\delta_n < \Xi$.
Then
$( e^\epsilon_n - e^\delta_n) / ( d^\epsilon_n - d^\delta_n)  > 0$ 
and also
$| e^\epsilon_n - e^\delta_n| < \varphi( |d^\epsilon_n - d^\delta_n|) $,
so $p^\delta \compat p^\epsilon$.
\end{proof}

Our forcing conditions will contain, among other things,
a finite $\sigma \subseteq D \times E$ which is a partial isomorphism; this
$\sigma$ will be a sub-function of the $f$ of Theorem \ref{thm-almost-c1}.
We let $g_0(x) = x^2 / (x^2 + 1)$, so
that $f_0(x) = x - \arctan(x)$.
The forcing conditions will determine successively
$\psi_0, \theta_0, \psi_1, \theta_1, \ldots$, and hence also $g_1, f_1, g_2, f_2, \ldots$.
We shall demand that all $\psi_n, \theta_n \in M_1$
(and hence also all $g_n, f_n \in M_1$), 
so that there are only countably many possibilities for them;
this will facilitate the proof that the poset is ccc. 
Then, $\lim_n f_n = f \supset \sigma$; the
$f_n$ will not actually extend $\sigma$;
rather, they will approximate $\sigma$ in the sense of the
following definition:

\begin{definition}
\label{def-correct}
$(\tau,g,f, \iota)$ is \emph{correctable} iff:
\begin{itemizz}
\item[\u P1.] $(0,0) \in \tau$ and
$\tau \in [ \RRR \times \RRR ]^{<\omega}$.
\item[\u P2.] $\tau$ is an order-preserving bijection.
\item[\u P7.] $g \in C(\RRR, [0,\infty) )$ and $g\iv\{0\} = \{0\}$.
\item[\u P8.]  $f(x) = \int_0^x g(t) \, dt$.
\item[\u P13.] $\iota > 0$ and whenever
$(d_0, e_0), (d_1, e_1) \in \tau$ and $d_0 < d_1$:
\[
0
\;<\; \frac{e_1 - e_0}{d_1 - d_0}  -
\frac{f(d_1) - f(d_0)}{d_1 - d_0}  \;<\;  \iota \ \ .
\]
\end{itemizz}
\end{definition}

The labels on these items correspond to the labels in Definition
\ref{def-poset} (of $\PPP$).
In $\PPP$, the $f,g$ will be replaced by suitable $f_n,g_n$.

Think of $\iota$ as being ``very small''.
So,  our hypotheses (\u P2)(\u P3)(\u P7)(\u P13) imply that
$f$ and $\tau$ are strictly increasing,
and between $d_0,d_1 \in \dom(\tau)$, the slope of
$f$ is very slightly less than the slope of $\tau$.

We remark that it is sufficient to assume that 
(\u P13) holds between adjacent elements of
$\dom(\tau)$; that implies the full (\u P13), since
if $d_0 < d_1 < d_2$ we have
\[
\begin{array}{ll}
0 \;<\; \big(\tau(d_1) -\tau(d_0) \big) - \big(f(d_1) -f(d_0) \big)
\;<\; \iota(d_1 - d_0)       &  \& \\
0 \;<\; \big(\tau(d_2) -\tau(d_1) \big) - \big(f(d_2) -f(d_1) \big)
\;<\; \iota(d_2 - d_1)       &  \Longrightarrow \\
0 \;<\; \big(\tau(d_2) -\tau(d_0) \big) - \big(f(d_2) -f(d_0) \big)
\;<\; \iota(d_2 - d_0)    \ \ \ .
\end{array}
\]

Since $f(0) = \tau(0) = 0$, we can set
$d_0 = 0$ or $d_1 = 0$ in (\u P13) to obtain, for
$(d,e) \in \tau$:
\[
\text{\it \u P12.}
\; \;  d,e > 0 \,\to\,  0 < (e - f(d)) <  \iota d \quad ; \quad
d,e < 0 \,\to\,  0 > (e - f(d)) >  \iota d  \ \ .
\]
That is, if $(d,e) \in \tau$,
then $f(d)$ is a slight under-estimate of $e$ when $d > 0$
and a slight over-estimate of $e$ when $d < 0$.
The next lemma says that this ``error'' can be corrected by
adding a small positive function $\theta$ to $g$:

\begin{lemma}
\label{lemma-correct}
Assume that $(\tau,g,f, \iota)$ is correctable and $J \subset \RRR$
is finite.  Then for some $\theta : \RRR \to \RRR$:
\begin{itemizz}
\item[a.]  $\theta(x) \ge 0$ for all $x$, and 
$\|\theta\| < \iota$, and $\theta(x) \to 0$ as $x \to \pm\infty$.
\item[b.] $\theta$ is continuous, and $\theta(d) = 0$ for all
$d \in J$.
\item[c.]  If $g^* = g + \theta$
and $f^*(x) = \int_0^x g^*(t) \, dt$, then $f^*(d) = e$ for
each $(d,e) \in \tau$.
\end{itemizz}
\end{lemma}
\begin{proof}
Since $\tau(0) = f(0) = 0$, item (c) will hold if we have,
for adjacent $d_0, d_1 \in \dom(\tau)$ with $d_0 < d_1$:
\[
\int_{d_0}^{d_1} \theta(t) \, dt =
(\tau(d_1) - \tau(d_0)) - ( f(d_1) - f(d_0) ) \ \  , 
\]
and this quantity is assumed to lie in $(0, \iota (d_1 - d_0))$.
It is now easy to construct a $C^\infty$ function $\theta$ which
satisfies this, along with (a)(b).
\end{proof}

\begin{definition}
\label{def-poset}
$\PPP$ is the set of all tuples
$p = (\sigma^p, N^p, g^p_{n+1}, f^p_{n+1},
\psi^p_n, \theta^p_n )_{n < N^p}$,
satisfying the following conditions.  We drop the superscript $p$ when
it is clear from context.
Let $g_0(x) = x^2 / (x^2 + 1)$ and
$f_0(x) = x - \arctan(x)$.
\begin{itemizz}
\item[P1.] $\sigma^p \in
[ D  \times E ]^{<\omega}$ and $(0,0) \in \sigma^p$.
\item[P2.] $\sigma^p$ is an order-preserving bijection.
\item[P3.] For $(d,e) \in \sigma^p$ and $(d,e) \ne (0,0)$:
$\hgt(e) < \hgt(d) < \hgt(e) + \omega$.
\item[P4.]  If $(d_0,e_0),(d_1,e_1) \in \sigma^p$ and $d_0 \ne d_1$ then 
$\hgt(d_0) \ne \hgt(d_1)$.
\item[P5.] $N^p \in \omega$.
\item[P6.] $\hgt(g_n) =  \hgt(f_n) =  \hgt(\psi_n) = \hgt(\theta_n) = 0$.
\item[P7.] $g_n \in C(\RRR, [0,2 - 2^{-n}) )$ 
and $g_n\iv\{0\} = \{0\}$ and $\lim_{x \to \pm \infty} g_n(x) = 1$.
\item[P8.]  $f_n(x) = \int_0^x g_n(t) \, dt$.
\item[P9.] $g_{n+1} = g_n - \psi_n + \theta_n$ when $n < N^p$.
\item[P10.] $\psi_n \in \AP_4$.
\item[P11.] $\theta_n \in C(\RRR)$ and $\| \theta_n \| \le  2^{-n - 1 }$.
\item[P12.] For $(d,e) \in \sigma$:
$d,e > 0 \,\to\,  0 < (e - f_{N^p}(d)) <  2^{-{N^p} - 2} d$ and \\
$d,e < 0 \,\to\,  0 > (e - f_{N^p}(d)) >  2^{-{N^p} -2} d$.
\item[P13.] Whenever
$(d_0, e_0), (d_1, e_1) \in \sigma$ and $d_0 < d_1$:
\[
0
\;<\; \frac{e_1 - e_0}{d_1 - d_0}  -
\frac{f_{N^p}(d_1) - f_{N^p}(d_0)}{d_1 - d_0}  \;<\;  2^{-{N^p} -2} \ \ .
\]
\end{itemizz}
Define $q \le p$ iff
\begin{itemizz}
\item[Q1.] $\sigma^q \supseteq \sigma^p$ and $N^q \ge N^p$.
\item[Q2.] 
$ ( g^p_{n+1}, f^p_{n+1}, \psi^p_n, \theta^p_n )=
 (g^q_{n+1}, f^q_{n+1}, \psi^q_n, \theta^q_n ) $
for all $n < N^p$. 
\item[Q3.] 
Whenever $(0,0) \ne (d,e) \in \sigma^p$ and $N^p < n \le N^q$:
$g^q_n(d) \in (0, 2^{-n})$.
\end{itemizz}
Then $\one = (\{(0,0)\}, 0)$; that is, when $N^p = 0$,
the rest of the tuple is empty.
\end{definition}

We shall now prove a sequence of lemmas leading
up to Theorem \ref{thm-almost-c1}, at the same time
explaining some of the clauses in Definition \ref{def-poset}.

The restriction on heights in (P3)(P4) will be important in the proof of ccc,
and are analogous to the restrictions in Lemma \ref{lemma-compat}.

If $G$ is a  generic filter on $\PPP$,  then in $\V[G]$ we can define
$\widehat \sigma =  \bigcup \{\sigma^p : p \in G\} $.
Then $ \widehat\sigma$ is an order-preserving function from
a subset of $D$ to a subset of $E$, and the $f$ of Theorem \ref{thm-almost-c1}
will extend $\widehat \sigma$ (Lemma  \ref{lemma-extend} below).

We shall apply Lemma \ref{lemma-limit-D} in $\V[G]$ to obtain
$f,g$, and (Q3) will let us prove that $g(d) = 0$ for all
$d \in \dom(\widehat \sigma)$.
Note that (Q3) is vacuous when $N^p = N^q$.

By (P1)(P2)(P7)(P8)(P13), each
$(\sigma , g_N, f_N, 2^{-N-2})$ is correctable.
Then, as noted above, (P12) follows, but we state it
separately for emphasis, since it is used to prove that
$f$ extends $\widehat \sigma$.
Also, the $g_n(x) <  2 - 2^{-n} $  asserted by (P7) follows
by induction from the other assumptions; specifically,
$g_0(x) < 1$, 
$g_{n+1} = g_n - \psi_n + \theta_n$,
$\theta_n(x) \le 2^{-n-1}$, and $\psi_n(x) \ge 0$.

\begin{definition}
\label{def-close}
$\mu(p) =
\min\{ |d_0 - d_1| : d_0, d_1 \in \dom(\sigma_p)
 \ \&\ d_0 \ne d_1\}$.
Call  a map $\zeta$ from  $\PPP$ into the rationals
a $\PPP$--\emph{function} iff
$\zeta(p) \in (0, \mu(p)/2 )$ for all $p$.
For such a $\zeta$, say $p,q \in \PPP$ are $\zeta$--\emph{close} iff
$N^p = N^q$ and $|\sigma^p| = |\sigma^q|$ and
$\zeta(p) = \zeta(q)$ and
all elements of $\dom(\sigma^p) \cup \dom(\sigma^q)$ have different heights and
\[
(g^p_{n+1}, f^p_{n+1}, \psi^p_n, \theta^p_n )_{n < N} =
(g^q_{n+1}, f^q_{n+1}, \psi^q_n, \theta^q_n )_{n < N} \ \ ,
\]
and, setting
$\zeta = \zeta(p) = \zeta(q)$:
For all $d \in \dom(\sigma^p)$
there is a $d' \in \dom(\sigma^q)$ such that $|d - d'| < \zeta$;
furthermore, if $(d,e) \in \sigma^p$ and $(d',e') \in \sigma^q$,
then $d = d'$ implies $e = e'$, and
$d \ne d'$ implies $0 < (e - e')/(d - d') < \zeta$. 
\end{definition}

Note that $p$ is always $\zeta$--close to itself.

The requirement that $\zeta(p) < \mu(p)/2 $ implies that the $d'$ 
above is uniquely determined from $d$.
The actual $\zeta(p)$ used in Lemma \ref{lemma-close-compat}
will be \emph{much} smaller than $\mu(p)/2 $.
The requirement that all the slopes  $(e - e')/(d - d') $
be small but positive will be
fulfilled in the proof of ccc using Lemma \ref{lemma-compat}.

If $p,q$ are $\zeta$--close, then they are ``close'' to being compatible,
with the tuple $ (\sigma^p \cup \sigma^q,  N^p, g^p_{n+1}, f^p_{n+1}, \psi^p_n, \theta^p_n )_{n < N}$
being a common extension, \emph{except} that this may fail (P2)(P12)(P13).

\begin{lemma}
\label{lemma-close-compat}
There is a $\PPP$--function $\zeta$ such that for all $p,q \in \PPP$:
If $p,q$ are $\zeta$--close then $p \not\perp q$ and there
is an $s \in \PPP$ such that $s \le p$ and $s \le q$ 
and $N^s = N^p + 1$.
\end{lemma}

We shall prove this later, after listing some of its consequences.
First, when $p = q$, we get:

\begin{corollary}
\label{cor-bigger-N}
For each $p \in \PPP$, there is an $s \le p$ with $N^s = N^p + 1$.
Hence, $\{q : N^q > i\}$ is dense in $\PPP$ for each $i$.
\end{corollary}

So, in $\V[G]$, we have $g_n, f_n, \psi_n, \theta_n$ for each $n \in \omega$;
e.g., $g_n = g_n^p$ for some (any) $p \in G$ such that $N^p \ge n$.
Then Lemma \ref{lemma-limit-D} applies:  
(1) is obvious,
(2) follows from (P11),
(3) follows from (P10), and (4) follows from (P7)(P9),
So, by Lemma \ref{lemma-limit-D},
$\langle g_n : n \in \omega \rangle$ converges pointwise
to some $g : \RRR \to [0, \infty)$, and $g \in \DD$;
also, $\| g \| \le 2$ by (P7).  Then, since the $g_n$ are uniformly bounded,
$\langle f_n : n \in \omega \rangle$ converges pointwise
to $f$, where $f(x) = \int_0^x g(t) \, dt$.

Regarding (Q3):  By not requiring
$g_n(d) \approx 0$ for \emph{all} $n \le N$,
we make it easier to
add new pairs $(d,e)$ into extensions of $p$
(see the proof of Lemma \ref{lemma-big-ran}).
Likewise, we only require (P12)(P13) for $n = N^p$, so that
when proving Lemma \ref{lemma-big-ran}, we do not need
to consider (P12)(P13) for $n < N^p$.
But still,

\begin{lemma}
\label{lemma-extend}
For $(d,e) \in  \widehat\sigma$: $g(d) = 0$ and $f(d) = e$.
\end{lemma}
\begin{proof}
Since $\langle g_n : n \in \omega \rangle$
and $\langle f_n : n \in \omega \rangle$
converge pointwise, it is sufficient to show that
some subsequence of  $\langle g_n(d) : n \in \omega \rangle$ converges to $0$
and 
some subsequence of  $\langle f_n(d) : n \in \omega \rangle$ converges to $e$.
Say $(d,e) \in p \in G$. 
Then by Corollary \ref{cor-bigger-N},
$S:= \{N^q : q \in G \, \wedge \, (d,e) \in q\}$ is infinite.
Then, applying (Q3)(P12), 
$\langle g_n(d) : n \in S \rangle$ converges to $0$ and 
$\langle f_n(d) : n \in S \rangle$ converges to $e$.
\end{proof}

Another consequence of Lemma \ref{lemma-close-compat}:

\begin{lemma}
\label{lemma-ccc}
$\PPP$ has the ccc.
\end{lemma}
\begin{proof}
Let $A \subseteq \PPP$ be uncountable; we prove that $A$ cannot be an
antichain.  Let $\zeta(p)$ be as in Lemma \ref{lemma-close-compat}.
We may assume that $\zeta(p) $ is the same rational
$\zeta$ for all $p \in A$.  Furthermore, by a  delta system argument,
we may assume that $A = \{p^\alpha : \alpha < \omega_1\}$
and $\sigma^{p^\alpha} = \sigma^\alpha \cup \tau$, where $\tau$
is the root of the delta system.
We may also assume (applying (P2)(P3)) that
the $\sigma^\alpha$ satisfy the hypotheses
of Lemma \ref{lemma-compat}, and that all $p^\alpha, p^\beta$
satisfy everything in
Definition \ref{def-close} (of  ``$\zeta$--close'')  except
possibly for the requirement
``$d \ne d'$ implies $0 < (e - e')/(d - d') < \zeta$''.
But now Lemma \ref{lemma-compat} implies that
there is \emph{some} pair $p^\alpha, p^\beta $ with $\alpha \ne \beta$
satisfying this requirement, so that $p^\alpha \not\perp q^\alpha$ by Lemma
\ref{lemma-close-compat}. 
\end{proof}

By applying Lemma \ref{lemma-correct} to $\PPP$  we get:

\begin{lemma}
\label{lemma-correction}
Fix $p \in \PPP$ and a finite $F \subset \RRR$.
Let $N = N^p$ and $\sigma = \sigma^p$.
Then for some $\theta : \RRR \to \RRR$:
\begin{itemizz}
\item[a.]  $\theta(x) \ge 0$ for all $x$, and 
$\|\theta\| < 2^{-N-2}$, and $\theta(x) \to 0$ as $x \to \pm\infty$.
\item[b.] $\theta$ is continuous, and $\theta(d) = 0$ for all
$d \in  F$.
\item[c.]  If $g^* = g_N + \theta$
and $f^*(x) = \int_0^x g^*(t) \, dt$, then $f^*(d) = e$ for
each $(d,e) \in \sigma$.
\end{itemizz}
\end{lemma}

\begin{lemma}
\label{lemma-big-ran}
$\ran(\widehat \sigma) = E $.
\end{lemma}
\begin{proof}
It is sufficient to prove that for each $e \in E$,
$\{q : e \in \ran(\sigma^q)\}$ is dense.
So fix $p \in \PPP$ with $e \notin \ran(\sigma^p)$, and we
find a $q \le p$ with $e \in \ran(\sigma^q)$;
$q$ will be exactly like $p$,
except that $\sigma^q  = \sigma^p \cup \{(d,e)\}$,
where $d \in D_\xi := \{d \in D : \hgt(d) = \xi\}$
and $\hgt(e) < \xi < \hgt(e) + \omega$
and $\xi$ is different from $\hgt(d')$ for all $d' \in \dom(\sigma^p)$.
Then $q \le p$ is clear, but we must make sure that $q \in \PPP$.

Let $f^*$ be as in Lemma \ref{lemma-correction}.
Then $f^*$ is a continuous increasing function, and,
using the $\lim_{x \to \pm \infty} g_N(x) = 1$ from (P7),
$f^*(x) \to \infty$ as $x \to \infty$ and
$f^*(x) \to -\infty$ as $x \to -\infty$. 
There is thus a unique $\hat d$ such that $f^*(\hat d) = e$.
For all $d$ sufficiently close to $\hat d$, setting
$\sigma^q  = \sigma^p \cup \{(d,e)\}$ will satisfy 
(P2)(P12)(P13), so choose such a $d$ in $D_\xi$,
which is possible because $D_\xi$ is dense.
\end{proof}

Although $\dom(\widehat \sigma) \ne D$ (by  (P3)(P4)), we do have:

\begin{lemma}
\label{lemma-big-dom}
In $\V[G]$, $\dom(\widehat \sigma)$ is an $\aleph_1$--dense subset of $D$.
\end{lemma}
\begin{proof}
Use the facts that $f$ is strictly increasing and continuous,
$\|f'\|  < \infty$
(by $P7$),  $f \supset \widehat\sigma$ (by Lemma \ref{lemma-extend}),
and $\V$, $\V[G]$ have the same $\aleph_1$ (by the ccc).
\end{proof}

We  are now done if we prove  Lemma \ref{lemma-close-compat}.
First, a few remarks.

As  noted above, to prove that $p \compat q$ whenever 
$p,q$ are $\zeta$--close, we need to make sure that the common extension
satisfies (P2)(P12)(P13).  But (P12) is a special case of (P13),
and it is easy to satisfy (P2); that is, if the function $\zeta$ is small enough
then $\sigma^p \cup \sigma^q$ will be order-preserving.
A  more serious issue is that the natural extension,
$(\sigma^p \cup \sigma^q,  N, g_{n+1}, f_{n+1},
\psi_n, \theta_n )_{n < N}$ may fail condition (P13);
that is, $(\sigma^p \cup \sigma^q,  g_N, f_N, 2^{-N-2})$
may not be correctable, since this
puts a \emph{lower} bound on the slopes between adjacent elements of
$\sigma$ in terms of the slope of $f_N$.
But here, the slopes between neighboring pairs
$(d,e)$ and $(d',e')$
are \emph{small} (bounded above by $\zeta$).

The common extension $s$
will have $\sigma^s = \sigma^p \cup \sigma^q$ but $N^s = N + 1$.
Then $\psi^s_N$ will be a linear combinations of functions of the form
$(1 + r|x-\bar d|)^{-1/2}$ for $\bar d$ close to a $d,d'$ pair and suitably 
large $r$.  Also, $r \approx 1/\sqrt{\zeta}$, so
$r \zeta \ll 1$, so that for $x$ near $d,d',\bar d$:
$g_N(x)$ and $\psi^s_N(x)$ will be approximately constant and
$g_N(x) - \psi^s_N(x)$ will be very slightly negative.
But $r$ will be large enough that $\int_0^d \psi(t) \,dt$
will be negligible for each $d$.

Now, we need to
define $\psi_N = \psi^s_N$  and $\theta_N = \theta^s_N$, which
will determine $g_{N+1} = g^s_{N+1}$ and $f_{N+1} = f^s_{N+1}$.
We do not know a ``simple'' definition of $\zeta$ which ``works'',
so rather than defining $\zeta$ right away,
we shall simply define $\psi_N = \psi^s_N$  and $\theta_N = \theta^s_N$,
and check that they have the right properties, assuming that $\zeta$ is small enough.
$\psi_n$ and $\theta_N$ will determine
$g_{N+1} = g^s_{N+1}$ and $f_{N+1} = f^s_{N+1}$
by $g_{N+1} = g_N - \psi_N + \theta_N$ and $f_{N+1}(x) = \int_0^x g_{N+1}(t)\, dt$.
We shall also have $\theta_{N} = \theta_N^\dag + \theta_N^*$
because there are
two tasks for $\theta_{N}$:  to make sure that
$g_{N+1} $ is positive (the task of  $\theta_N^\dag$), and to correct 
$f_{N+1} $ to come close to $\sigma^s$,
so as to satisfy (P13) (the task of $\theta_N^*$).  
Both $\theta_N^\dag$ and  $ \theta_N^*$ will be positive functions.

First, some notation:
Applying the definition of ``close'', 
let $L =   |\sigma^p| = |\sigma^q|$ and let
$\sigma^p = \{(d^p_\ell, e^p_\ell) : \ell < L\}$ and
$\sigma^q = \{(d^q_\ell, e^q_\ell) : \ell < L\}$, where
$| d^p_\ell - d^q_\ell | < \zeta$, which implies also
$| e^p_\ell - e^q_\ell | < \zeta^2$.

Before defining anything, we must make sure that $\sigma^s$ satisfies (P2);
that is, that  $\sigma^p \cup \sigma^q$ is an order-preserving bijection.
In view of the definition of ``close'', the problem is to show
that whenever $d^p_\ell  < d^p_j$  (and hence also $d^q_\ell  < d^q_j$), we have
both $e^p_\ell  < e^q_j$  and $e^q_\ell  < e^p_j$.
Since $| e^p_\ell - e^q_\ell | < \zeta^2$ and $| e^p_j - e^q_j | < \zeta^2$,
it is sufficient that 
$\zeta^2 <  |e^p_j - e^p_\ell | /3$ and $\zeta^2 <  |e^q_j - e^q_\ell | /3$;
but this follows if we assume that
$3 (\zeta(p))^2 < |e - e'|$ whenever $e,e' \in \ran(\sigma^p)$ and $e \ne e'$.

Next, we define $\psi_N$ so that  for each $\ell$, the function $g_N - \psi_N$ 
is slightly negative near $d^p_\ell $ and $ d^q_\ell $.   To make sure that $\psi_N \in M_1$:
Choose rational $\bar d_\ell$ such that 
$| \bar d_\ell - d^p_\ell |, | \bar d_\ell - d^q_\ell | < \zeta$.
Let
$\bar \gamma_\ell =
\max( g_N(  d^p_\ell ), g_N( d^q_\ell ), g_N( \bar d_\ell ) )$.
By (P7), $0 < \bar \gamma_\ell < 2 - 2^{-N}$.  Then choose
a rational $r$ such that $1/\sqrt{\zeta} < r < 2/\sqrt{\zeta}$ and
rational $\gamma_\ell$ so that $\bar \gamma_\ell < \gamma_\ell < 2 - 2^{-N}$
and  $\gamma_\ell - \bar \gamma_\ell <  2^{-N}/256$,
and define:
\[
\psi_N(x) =
\sum_{\ell < L}  \big(\gamma_\ell + 2^{-N}/16\big)
\big(1 + r |x - \bar d_\ell|\big)^{- 1/2} \ \ .
\]
Then $\psi_N \in \AP_4$ by Lemma \ref{lemma-psi-4}.
Clearly, $g_N(\bar d_\ell) - \psi_N(\bar d_\ell) <
\gamma_\ell - (\gamma_\ell +  2^{-N}/16) < 0$,
but we wish to assert also that
$g_N(x) - \psi_N(x) < 0$ whenever $|x - \bar d_\ell| < \zeta$; in particular,
for $x = d^p_\ell , d^q_\ell $.  We may assume that always $\zeta(p) < 1$; then,
for $|x - \bar d_\ell| < \zeta$:
\[
\big(1 + r |x - \bar d_\ell|\big)^{- 1/2} \ge
(1 +  \big(2/\sqrt{\zeta}\big) \cdot \zeta)^{- 1/2}  =
(1 +  2\sqrt{\zeta})^{- 1/2}  > 1 - \sqrt{\zeta} \ \ .
\]
Now, assume that our function $\zeta(p)$ satisfies
$
\forall d \in \dom(\sigma^p) \, \forall x \, [
|x - d | < 2\zeta(p) \to  |g_N(x) - g_N(d)| <  2^{-N}/256] 
$.
Then, using $\bar \gamma_\ell <  \gamma_\ell$:
\[
g_N(x) - \psi_N(x) \le
\big( \gamma_\ell +  2^{-N}/256\big) - 
\big(\gamma_\ell + 2^{-N}/16\big)
\big(1  - \sqrt{\zeta} \big) 
\]
when $|x - \bar d_\ell| < \zeta$, so that
\[
g_N(x) - \psi_N(x) \le
 2^{-N}/256 -  2^{-N}/16 + \big(\gamma_\ell + 2^{-N}/16\big) \sqrt{\zeta}
< 0 \ \ ;
\]
This last $<$ holds if we assume that always $\sqrt{\zeta(p)} <  2^{-N}/256$.

Now, we define
$\theta_N^\dag(x) = \max(0, \psi_N(x) - g_N(x)) + \varepsilon x^2 / (x^2 + 1)$, 
where $\varepsilon $ is a positive rational which is small 
enough to make the following argument work.
Let $g_N^\dag(x) := g_N(x) - \psi_N(x) + \theta_N^\dag(x)$,
which is positive everywhere except  at $0$.
Let $f_N^\dag(x)  = \int_0^x g_N^\dag(t)  \, dt$.
We plan to show that
$(\sigma^p \cup \sigma^q,  g_N^\dag, f_N^\dag, 2^{-N-2})$
is correctable.

(P11) requires $\|\theta_N\| \le 2^{-N-1}$.  
To accomplish this, we first verify that
$\|\theta_N ^\dag\| \le 2^{-N-2}$, and later we shall verify that
$\|\theta_N ^*\| \le 2^{-N-2}$.
 As long as $\varepsilon \le 2^{-N-2}$,
$\theta_N ^\dag(x) \le 2^{-N-2}$ whenever $\psi_N(x) \le g_N(x)$,
which holds as $x \to \pm\infty$ since $g_N(x) \to 1$ and $\psi_N(x) \to 0$.
Also, $\theta_N^\dag(x) \le  \psi_N(x)  + \varepsilon$,
so that if $\varepsilon \le 2^{-N-3}$,
then $\theta_N ^\dag(x) \le 2^{-N-2}$ whenever 
$\psi_N(x) \le 2^{-N-3}$,
and if $\zeta$ is large enough, this will hold unless $x$ is
very close to one of the $\bar d_\ell$.  
More precisely, if $|x - \bar d_\ell | \ge c$ for all $\ell $, then
$\psi_N(x) \le 2L \big(1 + r  c \big)^{- 1/2} < 2L r^{- 1/2}  c^ {- 1/2}  $.
Then, using $1/\sqrt{\zeta} < r < 2/\sqrt{\zeta}$,
if $|x - \bar d_\ell | \ge \zeta^{1/4}$ for all $\ell $ then
\[
\psi_N(x) < 
2L \zeta^{1/4} \zeta^{-1/8} = 2L \zeta^{1/8}  \ \ .
\]
Then  $\theta_N ^\dag(x) \le 2^{-N-2}$ for these $x$ provided we assume that
our $\zeta$ function satisfies
$ 2| \sigma^p| \cdot  (\zeta(p))^{1/8} \le 2^{-N^p-3}$.

Now, fix $x$ and assume that $|x - \bar d_m | \le \zeta^{1/4}$ for some $m$;
this $m$ will be unique if we assume that $(\zeta(p))^{1/4} < \mu(p)/4$
for all $p$.
We need to show that $\theta_N ^\dag(x) \le 2^{-N-2}$.
Assume that $\psi_N(x) > g_N(x)$, since we have already covered the case
that $g_N(x) \le \psi_N(x)$.  So, 
\begin{align*}
& \theta_N^\dag(x) = \psi_N(x) - g_N(x) + \varepsilon x^2 / (x^2 + 1) \le \\
& \qquad \varepsilon + 
\big(\gamma_m + 2^{-N}/16\big)\big(1 + r |x - \bar d_m|\big)^{- 1/2} + \\
& \qquad \sum_{\ell \ne  m}  \big(\gamma_\ell +
 2^{-N}/16\big)\big(1 + r |x - \bar d_\ell|\big)^{- 1/2}  -g_N(x) \le\\
&\varepsilon + \big(\gamma_m + 2^{-N}/16\big) + 2^{-N-4} - g_N(x) =
\varepsilon + (\gamma_m - g_N(x) ) + 2^{-N-2}/2 \ \ ;
\end{align*}
for the last $\le$, use the previous argument, but now assuming that
our $\zeta$ function satisfies
$ 2| \sigma^p| \cdot  (\zeta(p))^{1/8} \le 2^{-N^p-4}$.
Assuming that $\varepsilon  \le  2^{-N -2}/4$ and
$\gamma_m - g_N(x)  \le  2^{-N -2}/4$
we have $\theta_N^\dag(x)  \le  2^{-N-2}$.
Since $\bar \gamma_m \in
\{ g_N(  d^p_m ), g_N( d^q_m ), g_N( \bar d_m ) \}$ and
$\bar \gamma_m < \gamma_m <  \bar \gamma_m + 2^{-N}/256 $,
and $x$ is within $2 \zeta^{1/4}$ of each of $d^p_m , d^q_m , \bar d_m$,
we obtain $\gamma_m - g_N(x)  \le  2^{-N -2}/4$ if we assume that
$\forall x \, \forall d \in \dom(\sigma^p) \,
[  |x - d| \le 2 (\zeta(p))^{1/4} \to
|g_N(x) - g_N(d)| \le 2^{-N -2}/16$.

We next show that
$(\sigma^p,  g_N^\dag, f_N^\dag, 2^{-N-2})$ and
$(\sigma^q,  g_N^\dag, f_N^\dag, 2^{-N-2})$ are correctable. 
To do this, we bound the change
in $f_N(d)$ caused by replacing $g_N$ by $g_N - \psi_N + \theta_N^\dag$;
this change is $\int_0^d (\psi_N(t) - \theta_N^\dag(t)) \, dt$.  Let $\Delta$
be the diameter of $\dom(\sigma^p)$.  Then,
since $\gamma_\ell + 2^{-N}/16 < 2$ and $r > 1/\sqrt{\zeta}$

\begin{align*}
& \int_0^d \psi_N(t)\, dt  <
2L \int_0^\Delta (1 + rt)^{-1/2} \, dt=
\frac{4L}{r} \left[ \sqrt{1 + r\Delta} - 1\right] \le \\
& \frac{4L}{r} \sqrt{r\Delta}  \le
4L \sqrt{\Delta} \sqrt[4]{\zeta} \  .
\end{align*}
This can be made arbitrarily small by requiring the $\zeta$
function to be small enough.
Likewise, $\int_0^d \theta_N^\dag(t)\, dt$
can be made arbitrarily small using
$0 \le \theta_N^\dag(t) \le \psi_N(t) + \varepsilon $.
So, the correctability of
$(\sigma^p,  g_N^\dag, f_N^\dag, 2^{-N-2})$ and
$(\sigma^q,  g_N^\dag, f_N^\dag, 2^{-N-2})$  follows from the correctability of
$(\sigma^p,  g_N, f_N, 2^{-N-2})$ and
$(\sigma^q,  g_N, f_N, 2^{-N-2})$  if $\zeta$ makes
$f_N^\dag$ close enough to $f_n$.

Now, to verify that
$(\sigma^p \cup \sigma^q,  g_N^\dag, f_N^\dag, 2^{-N-2})$
is correctable, we must show that
(\u P13) holds between adjacent elements of
$\dom( \sigma^p \cup \sigma^q )$.  There are two cases
not already covered by the above:

Case I:  Between $d^p_m$ and $d^q_\ell$ where $m \ne \ell$: We need
\[
0
\;<\; \frac{e_m^p - e_\ell^q}{d_m^p - d_\ell^q}  -
\frac{f_N^\dag(d_m^p) - f_N^\dag(d_\ell^q)}{d_m^p - d_\ell^q}  \;<\;
 2^{-N-2} \ \ .
\]
This is handled by making $\zeta$ small enough,
since the inequality holds if we replace
$d^q_\ell, e^q_\ell$ by $d^p_\ell, e^p_\ell$. 

Case II:  Between $d^p_\ell$ and $d^q_\ell$, when
$d^p_\ell \ne d^q_\ell$.
WLOG, $d^p_\ell < d^q_\ell$, and we need
\[
0
\;<\; \frac{e^q_\ell - e^p_\ell}{d^q_\ell - d^p_\ell}  -
\frac{f_N^\dag(d^q_\ell) - f_N^\dag(d^p_\ell)}{d^q_\ell - d^p_\ell}  \;<\;
 2^{-N-2} \ \ .
\]
By the definition of ``close'', we have
$0 < (e^q_\ell - e^p_\ell)/(d^q_\ell - d^p_\ell) < \zeta$,
and our assumptions above about $\zeta$ already imply that
$\zeta < 2^{-N-2}$.  Thus, it is sufficient to have
$f_N^\dag(d^q_\ell) - f_N^\dag(d^p_\ell) < e^q_\ell - e^p_\ell$.
Now we have already checked that $g_N(x) - \psi_N(x) < 0$ 
for $x \in [d^p_\ell , d^q_\ell]$, so that
$g_N^\dag(x) = \varepsilon x^2 / (x^2 + 1)$ for these $x$.
Then $f_N^\dag(d^q_\ell) - f_N^\dag(d^p_\ell) =
\varepsilon \int_{d^p_\ell}^{d^q_\ell} x^2 / (x^2 + 1) \, dx
< \varepsilon(d^q_\ell - d^p_\ell)$,
which will be less than $e^q_\ell - e^p_\ell$ if we have 
chosen a small enough $\varepsilon$.

Then, by Lemma \ref{lemma-correct}, there is a positive
function $\theta^{\#}_N$ such that $\| \theta^{\#}_N \| <  2^{-N-2} $
and, setting $g_{N+1}^\# = g_N^\dag + \theta^{\#}_N$ and integrating,
gives us $f^\#_{N+1} \supset \sigma^s$; so, instead of (P13) for $s$
we have, for $(d_0, e_0), (d_1, e_1) \in \sigma^s$ and $d_0 < d_1$:
\[
(e_1 - e_0) - (f^\#_{N+1}(d_1) - f^\#_{N+1}(d_0)) = 0 \ \ .
\]
This is not exactly what we want, and this $\theta^{\#}_N$ need not be in $M_1$,
but by modifying our  $\theta^{\#}_N$ slightly, we can get
$\theta^*_N \in M_1$ so that
setting $g_{N+1}^s = g_N^\dag + \theta^{*}_N$ and integrating
gives us $f^s_{N+1} $ satisfying
\[
 0 < (e_1 - e_0) - (f^s_{N+1}(d_1) - f^s_{N+1}(d_0)) <
 2^{-N-3}(d_1 - d_0) \ \ ,
\]
which is (P13) for the forcing condition $s$, so that $s \in \PPP$.

Of course, we also need to verify that $s \le p$ and $s \le q$.
(Q1) and (Q2) are trivial, but (Q3) requires 
$g_{N+1}^s(d) \in (0,  2^{-N-1})$
for $d \in \dom(\sigma^p) \cup \dom(\sigma^q) \backslash \{0\}$.
Now $g_{N+1}^s = g_N^\dag + \theta_N^*$, and we already know
that $g_N^\dag(d) = \varepsilon d^2 / (d^2 + 1) < \varepsilon$,
and we already assumed that $\varepsilon  \le  2^{-N -2}$.
So,  when we apply  Lemma \ref{lemma-correct},
get $\theta^{\#}_N(d) = 0$ for these $d$.
Then, when we modify $\theta^{\#}_N$ slightly 
to get $\theta^*_N $, make sure that
$\theta^{\#}_N(d) - \theta^*_N(d) \in (0, 2^{-N-2})$.

\end{document}